\documentclass[11pt]{amsart}
\usepackage{amssymb,amsmath,amsthm}
\theoremstyle{plain}
\newtheorem{theorem}{Theorem}[section]
\newtheorem{lemma}[theorem]{Lemma}

\theoremstyle{remark}
\newtheorem*{remarks}{Remarks}

\newcommand{\CC}{{\mathbb C}}
\newcommand{\DD}{{\mathbb D}}
\newcommand{\RR}{{\mathbb R}}

\newcommand{\cD}{{\mathcal D}}
\newcommand{\cH}{{\mathcal H}}
\DeclareMathOperator{\hol}{\mathrm Hol}
\DeclareMathOperator{\supp}{\mathrm supp}

\renewcommand{\hat}{\widehat}

\begin{document}

\title{Weakly multiplicative distributions and weighted Dirichlet spaces}
\date{27 Sep 2021}

\author{Javad Mashreghi}
\address{D\'epartement de math\'ematiques et de statistique, Universit\'e Laval,
Qu\'ebec City (Qu\'ebec),  Canada G1V 0A6.}
\email{javad.mashreghi@mat.ulaval.ca}

\author{Thomas Ransford}
\address{D\'epartement de math\'ematiques et de statistique, Universit\'e Laval,
Qu\'ebec City (Qu\'ebec),  Canada G1V 0A6.}
\email{thomas.ransford@mat.ulaval.ca}

\thanks{First author supported by an NSERC Discovery Grant. Second author  supported by grants from NSERC and the Canada Research Chairs program.}

\begin{abstract}
We show that if $u$ is a compactly supported distribution
on the complex plane such that, for every pair of entire functions $f,g$,
\[
\langle u,f\overline{g}\rangle=\langle u,f\rangle\langle u,\overline{g}\rangle,
\]
then $u$ is supported at a single point.
As an application, we complete the classification of  all weighted Dirichlet spaces on the unit disk that are 
de Branges--Rovnyak spaces by showing that, for such spaces, the weight is necessarily a superharmonic
function. 
\end{abstract}



\maketitle


\section{Introduction}\label{S:intro}

In the first part of this note, we study a family of distributions that possessing a 
certain weak multiplicativity property. 
For general background on distribution theory we refer to the book of Friedlander \cite{Fr98}.

Let $u$ be a compactly supported distribution on the complex plane $\CC$.
Thus $u$ is a continuous linear functional on $C^\infty(\CC)$,
where $C^\infty(\CC)$ is endowed with the Fr\'echet-space topology of uniform convergence of all
derivatives on compact sets.
In what follows, we write $\langle u,\psi\rangle$ for $u(\psi)$.

We say that $u$ is \emph{weakly multiplicative} if, for every pair of entire functions $f,g$, we have
\begin{equation}\label{E:fg}
\langle u,f\overline{g}\rangle=\langle u,f\rangle\langle u,\overline{g}\rangle.
\end{equation}
Clearly this implies that, for all integers $j,k\ge0$,
\begin{equation}\label{E:powers}
\langle u,z^j\overline{z}^k\rangle=\langle u,z^j\rangle\langle u,\overline{z}^k\rangle.
\end{equation}
Conversely, if \eqref{E:powers} holds for all $j,k\ge0$, then \eqref{E:fg} holds for
for all polynomials $f,g$,
and because  the Taylor series of an entire function converges to the function in the topology
of $C^\infty(\CC)$, it follows that \eqref{E:fg} holds for all entire functions $f,g$.

It is not hard to see that, if $u$ is strongly multiplicative in the sense that
$\langle u,\phi\psi\rangle=\langle u,\phi\rangle\langle u,\psi\rangle$ for all $\phi,\psi\in C^\infty(\CC)$,
then either $u=0$ or $u=\delta_a$, the Dirac distribution at some $a\in\CC$.
This is no longer true for weakly multiplicative distributions, though it is still the case that they
must be supported at a single point. This is the main thrust of our first theorem.
As usual, we write $\partial:=(1/2)(\partial_x-i\partial_y)$ and $\overline{\partial}:=(1/2)(\partial_x+i\partial_y)$
for the Cauchy--Riemann operators, and $\delta_a$ for the Dirac distribution at $a$.

\begin{theorem}\label{T:dist}
Let $u$ be a compactly supported distribution on $\CC$.
Then $u$ is weakly multiplicative if and only if $u=0$ or  $u=p(\partial)q(\overline{\partial})\delta_a$,
where  $a\in\CC$ and $p,q$ are complex polynomials with $p(0)=q(0)=1$.
\end{theorem}

This theorem is a generalization of a result of Youssfi \cite[Proposition 1]{Yo20},
who treated the special case of  measures $u$ that satisfy \eqref{E:powers}
(in which  case, the conclusion is simply that $u$ is zero or a  Dirac measure).
The extra generality afforded by distributions is important for the application that follows.

Youssfi's proof relies on a result of  Luecking \cite{Lu08}
characterizing  Toeplitz operators on the Bergman space that have finite rank.
Luecking's result applies only to Toeplitz operators with measure-valued symbols,
but his result was subsequently extended by Alexandrov and Rozenblum \cite{AR09}
to cover the case of distribution-valued symbols,
and in principle their result could be used to deduce Theorem~\ref{T:dist}.
However, the Toeplitz operator that arises in our particular situation actually has rank one, 
which permits a considerable simplification of these ideas, leading to a direct proof of Theorem~\ref{T:dist}.
This proof is presented in \S\ref{S:dist}.

In the second part of the note, we present an application of Theorem~\ref{T:dist}
to weighted Dirichlet spaces.
Let $\DD$ be the open unit disk and
let $dA$ be normalized area measure on $\DD$.
We shall call $\omega$ a \emph{weight} on $\DD$ if $\omega\in L^1(\DD,dA)$ and $\omega\ge0$.
The \emph{weighted Dirichlet space} $\cD_\omega$ is the set of holomorphic functions $f$ on $\DD$ such that
\[
\cD_\omega(f):=\int_\DD |f'(z)|^2 \omega(z)\,dA(z)<\infty.
\]
It is known that, for certain weights, $\cD_\omega$ is a de Branges--Rovnyak space.
What this means and why it is significant will be explained in \S\ref{S:Dirichlet}.
The superharmonic weights $\omega$ for which $\cD_\omega$ is a de Branges--Rovnyak space
were classified in \cite{EKKMR16}. The following result,
which is an application of Theorem~\ref{T:dist},
takes care of the case of non-superharmonic weights,
thereby completing the classification.

\begin{theorem}\label{T:Dirichlet}
If $\omega$ is a weight on $\DD$ such that $\cD_\omega$ is a de Branges--Rovnyak space,
then $\omega$ is (almost-everywhere equal to) a function that is superharmonic on $\DD$.
\end{theorem}

The theory of Dirichlet spaces with harmonic weights was developed by 
Richter \cite{Ri91} and Richter--Sundberg \cite{RS91},
and extended to the case of superharmonic weights by Aleman \cite{Al93}.
There are  now many beautiful results in this area,
notably Shimorin's theorem \cite{Sh02} that $\cD_\omega$ has a complete Pick kernel
whenever $\omega$ is a superharmonic weight.
Theorem~\ref{T:Dirichlet} may be viewed as a further vindication that 
the superharmonic weights form a natural class.
The theorem is proved in \S\ref{S:Dirichlet}.


\section{Weakly multiplicative distributions}\label{S:dist}

To prove Theorem~\ref{T:dist},
we first establish a density result for $C^\infty(\CC^2)$.
As usual, we endow this space with the Fr\'echet-space topology 
of uniform convergence of all derivatives on compact subsets.

\begin{lemma}\label{L:density}
Functions of the form $|r(z_1,z_2)|^2$, where $r$ is a polynomial,
span a dense subspace of $C^\infty(\CC^2)$.
\end{lemma}

\begin{proof}
Let $v$ be a compactly supported distribution on $\CC^2$ such that,
for all polynomials $r(z_1,z_2)$,
\[
\langle v,\,|r|^2\rangle=0.
\]
By polarization, it follows that,
for all polynomials $p(z_1,z_2)$ and $q(z_1,z_2)$, we have
\[
\langle v,\,p\overline{q}\rangle=0.
\]
Since the Taylor series of an entire function on $\CC^2$ converges to the function in the topology
of $C^\infty(\CC^2)$, we deduce that, for all entire functions $f,g$ on $\CC^2$,
\[
\langle v,\,f\overline{g}\rangle=0.
\]
In particular, taking $f(z_1,z_2):=e^{(b-ia)z_1/2+(d-ic)z_2/2}$ and  $g(z):=e^{-(b-ia)z_1/2-(d-ic)z_2/2}$, 
we see that, for all $a,b,c,d\in\RR$,
\[
\langle v,\, e^{-i(ax_1+by_1+cx_2+dy_2)}\rangle=0.
\]
This amounts to saying that the Fourier transform $\hat{v}$  of $v$ satisfies
$\hat{v}(a,b,c,d)=0$ for all $a,b,c,d\in\RR$,
whence also $v=0$.

We have shown that the only continuous linear functional on $C^\infty(\CC^2)$ vanishing
on all functions of the form $|r(z_1,z_2)|^2$ is the zero functional. By the Hahn--Banach theorem,
functions of this form span a dense subspace of $C^\infty(\CC^2)$.
\end{proof}

\begin{proof}[Proof of Theorem~\ref{T:dist}]
If $u=p(\partial)q(\overline{\partial})\delta_a$ with $p(0)=q(0)=1$, then,
for every pair of entire functions $f,g$, we have
\[
\langle u,f\overline{g}\rangle
=(q(-\overline{\partial})p(-\partial)f\overline{g})(a)
=(p(-\partial)f)(a).(q(-\overline{\partial})\overline{g})(a)
=\langle u,f\rangle\langle u,\overline{g}\rangle,
\]
and thus \eqref{E:fg} holds. Obviously \eqref{E:fg} also holds if $u=0$.

We now turn to the converse.
Let $u$ be a compactly supported distribution on $\CC$ such that 
\eqref{E:fg} holds for every pair of entire functions $f,g$.
Consider the tensor product $u\otimes u$. By definition, this is the unique 
compactly supported distribution on $\CC^2$ such that
\begin{equation}\label{E:tensor}
\langle u\otimes u,\, \psi_1(z_1)\psi_2(z_2)\rangle
=\langle u,\psi_1\rangle\langle u,\psi_2\rangle
\quad(\psi_1,\psi_2\in C^\infty(\CC)).
\end{equation}
Let  $j,k,m,n$ be non-negative integers, and consider the expression
\[
\Bigl\langle u\otimes u,~ (z_1-z_2)(z_1^jz_2^k+z_1^kz_2^j)\overline{z}_1^m\overline{z}_2^n\Bigr\rangle.
\]
Expanding this out using \eqref{E:tensor} and then \eqref{E:powers}, we find that it is equal to $0$.
Any polynomial $s(z_1,z_2)$ that is symmetric  
(i.e.\ $s(z_1,z_2)=s(z_2,z_1)$) is a linear combination of polynomials of the form
$(z_1^jz_2^k+z_1^kz_2^j)$. Hence, given a symmetric polynomial $s(z_1,z_2)$ and an arbitrary polynomial $t(z_1,z_2)$, we have
\[
\Bigl\langle u\otimes u,~(z_1-z_2)s(z_1,z_2)\overline{t(z_1,z_2)}\Bigr\rangle=0.
\]
In particular, taking $t(z_1,z_2):=(z_1-z_2)s(z_1,z_2)$, 
we deduce that, for all symmetric polynomials $s(z_1,z_2)$,
\begin{equation}\label{E:sympoly}
\Bigl\langle u\otimes u,~|z_1-z_2|^2|s(z_1,z_2)|^2\Bigr\rangle=0.
\end{equation}

Our goal now is to show that the support of the distribution $u\otimes u$ is contained in
the diagonal set $\Delta:=\{(z_1,z_2)\in \CC^2: z_1=z_2\}$. 
Let $\psi\in C^\infty(\CC^2)$ be a function such that $\psi=0$ on an open neighborhood of $\Delta$.
Define $\rho:\CC^2\to\CC$ by
\begin{align*}
\rho(w_1,w_2):=&\psi\Bigl(\frac{w_1}{2}+\frac{1}{2}\sqrt{w_1^2-4w_2},~\frac{w_1}{2}-\frac{1}{2}\sqrt{w_1^2-4w_2}\Bigr)\\
&+\psi\Bigl(\frac{w_1}{2}-\frac{1}{2}\sqrt{w_1^2-4w_2},~\frac{w_1}{2}+\frac{1}{2}\sqrt{w_1^2-4w_2}\Bigr).
\end{align*}
The symmetry in the definition ensures that $\rho$ is well defined, 
and the fact that $\psi$ vanishes on a neighborhood of $\Delta$
ensures that $\rho\in C^\infty(\CC^2)$.
By Lemma~\ref{L:density}, $\rho(w_1,w_2)$ is the limit in $C^\infty(\CC^2)$ of finite linear combinations
of functions of the form $|r(w_1,w_2)|^2$ where $r$ is a polynomial.
Therefore $\rho(z_1+z_2,~z_1z_2)$ is the limit in $C^\infty(\CC^2)$ of finite linear combinations
of functions of the form $|r(z_1+z_2,\,z_1z_2)|^2$ where $r$ is a polynomial.
Since $r(z_1+z_2,\,z_1z_2)$ is a symmetric polynomial, \eqref{E:sympoly} implies that
\[
\Bigl\langle u\otimes u,~|z_1-z_2|^2|r(z_1+z_2,\,z_1z_2)|^2\Bigr\rangle=0.
\]
By linearity and continuity, it follows that
\[
\Bigl\langle u\otimes u,~|z_1-z_2|^2\rho(z_1+z_2,\,z_1z_2)\Bigr\rangle=0.
\]
However, a simple computation shows that, for all $z_1,z_2\in\CC$,
\[
\rho(z_1+z_2,\,z_1z_2)=\psi(z_1,z_2)+\psi(z_2,z_1),
\]
whence, using the symmetry of $u\otimes u$ and $|z_1-z_2|^2$, we obtain
\[
\Bigl\langle u\otimes u,~|z_1-z_2|^2\psi(z_1,z_2)\Bigr\rangle=0.
\]
As this holds for all $\psi\in C^\infty(\CC^2)$ that vanish on an  open neighborhood of $\Delta$,
we conclude that $\supp(u\otimes u)\subset\Delta$, as claimed.

Now it is well known that, in general, $\supp(u\otimes u)=\supp u\times \supp u$
(see e.g.\ \cite[Theorem~4.3.3\,(ii)]{Fr98}). Together with the inclusion $\supp(u\otimes u)\subset\Delta$,
this implies that either $\supp u$ is empty or that it consists of a single point in $\CC$.
If $\supp u$ is empty, then $u=0$. If $\supp u=\{a\}$ for some $a\in\CC$,
then, by \cite[Theorem~3.2.1]{Fr98},
$u$ must be a finite linear combination of derivatives of the Dirac distribution at $a$, 
so it can be written as
\[
u=\sum_{j,k\le N}c_{jk}\partial ^j\overline{\partial}^k\delta_a,
\]
for some choice of coefficients $c_{jk}\in\CC$. In this case, for all integers $m,n\ge0$,
\[
\langle u,\, (z-a)^m\overline{(z-a)^n}\rangle=(-1)^{m+n}m!n!c_{mn}.
\]
But also, by \eqref{E:fg}, we have
\[
\langle u,\, (z-a)^m\overline{(z-a)^n}\rangle=\langle u,\, (z-a)^m\rangle\langle u,\, \overline{(z-a)^n}\rangle.
\]
It follows that $c_{mn}=c_{m0}c_{0n}$ for all $m,n$. In particular, $c_{00}=c_{00}^2$, so  $c_{00}=0$ or $1$.
If $c_{00}=0$, then $c_{m0}=c_{0n}=0$ for all $m,n$,
whence $c_{mn}=0$ for all $m,n$, and so $u=0$. If $c_{00}=1$, then
\[
u=(\sum_{j\le N}c_{j0}\partial^j)(\sum_{k\le N} c_{0k}\overline{\partial}^k)\delta_a=p(\partial)q(\overline{\partial})\delta_a,
\]
where $p,q$ are polynomials with $p(0)=q(0)=1$. 
This completes the proof.
\end{proof}


\section{Weighted Dirichlet spaces and de Branges--Rovnyak spaces }\label{S:Dirichlet}

Given a holomorphic function $b:\DD\to\overline{\DD}$, the associated
\emph{de Branges--Rovnyak space}
$\cH(b)$ is the unique reproducing kernel Hilbert space with kernel
\[
\frac{1-b(z)\overline{b(w)}}{1-z\overline{w}}.
\]
It is always a subspace of the Hardy space $H^2$,
but not necessarily closed in $H^2$.
For background on de Branges--Rovnyak spaces, we refer to the books of
Sarason \cite{Sa94} and Fricain--Mashreghi \cite{FM16a,FM16b}.

In this section we address the following question:
\emph{For which weights $\omega$ is $\cD_\omega$ a de Branges--Rovnyak space?}
By this we mean that there exists a holomorphic function $b:\DD\to\overline{\DD}$
such that $\cD_\omega=\cH(b)$ as sets and also such that
\begin{equation}\label{E:isometry}
\|f\|_{\cH(b)}^2=\|f\|_{H^2}^2+\cD_\omega(f)
\quad(f\in\cD_\omega).
\end{equation}

The first examples of such  weights were given by Sarason \cite{Sa97}.
He showed that, for each of the weights
\begin{equation}\label{E:harmweights}
\omega_\zeta(z):=\frac{1-|z|^2}{|z-\zeta|^2} \quad(\zeta\in\partial\DD),
\end{equation}
the space $\cD_{\omega_\zeta}$ is a de Branges--Rovnyak space.
From this fact, he was able to deduce an interesting inequality for
the dilations $f_r(z):=f(rz)~(0<r<1)$ of a function $f\in\hol(\DD)$, namely that
$\cD_{\omega}(f_r)\le \cD_{\omega}(f)$, not just for $\omega=\omega_\zeta$,
but for all harmonic weights $\omega$ on $\DD$.

In view of the fact that this last inequality holds for all harmonic weights, 
one might wonder if indeed $\cD_\omega$ is a
de Branges--Rovnyak space for every harmonic weight on $\DD$.
The authors of \cite{CGR10} showed that this is actually \emph{not} the case. 
In fact, among harmonic weights, the only ones for which $\cD_\omega$ is a
de Branges--Rovnyak space are positive multiples of those displayed in \eqref{E:harmweights}.

The authors of \cite{EKKMR16} took this analysis one step further by examining
superharmonic weights. They showed that, within this class,
there is a new family of weights  
for which $\cD_\omega$ is a de Branges--Rovnyak space, namely
those of the form
\begin{equation}\label{E:superharmweights}
\omega_\zeta(z):=\log\Bigl|\frac{1-\overline{\zeta}z}{z-\zeta}\Bigr|
\quad(\zeta\in\DD).
\end{equation}
They further showed that, among superharmonic weights, 
the only ones for which $\cD_\omega$ is a
de Branges--Rovnyak space are positive multiples of those displayed in 
\eqref{E:harmweights} and \eqref{E:superharmweights}.

The article \cite{EKKMR16} concluded with the question of what happens in the case
of non-superharmonic weights. Theorem~\ref{T:Dirichlet} above answers this question.

To prove this theorem,
we need the following lemma, essentially taken from \cite{EKKMR16}.

\begin{lemma}\label{L:Dirichlet}
Let $\omega$ be a weight on $\DD$ such that $\|\omega\|_{L^1(\DD)}=1$.
If $\cD_\omega$ is a de Branges--Rovnyak space, then
there exists a holomorphic function $\phi$ on $\DD$ such that
\begin{equation}\label{E:phi}
\frac{|\phi(w)|^2}{1-|w|^2}=\int_\DD\frac{|w|^2}{|1-z\overline{w}|^4}\omega(z)\,dA(z)
\quad(w\in\DD).
\end{equation}
\end{lemma}

\begin{proof}
See the proof of \cite[Theorem~6.2]{EKKMR16}, and in particular the formula (6.2) in that proof.
\end{proof}

\begin{remarks}
(i) The function $\phi$ in Lemma~\ref{L:Dirichlet} is closely related to the function $b$
such that $\cD_\omega=\cH(b)$. In fact $\phi(z)=b(z)/a(z)$, where $a$ is the unique outer function
on $\DD$ such that $a(0)>0$ and $|a(e^{i\theta})|^2=1-|b(e^{i\theta})|^2$ a.e.\ on $\partial\DD$.
For more on this, see \cite{EKKMR16}. However, we do need this fact here.

(ii) The assumption that $\omega$ be normalized so that $\|\omega\|_{L^1(\DD)}=1$
is not a serious restriction. Indeed, if $\cD_\omega$ is a de Branges--Rovnyak space,
then so is $\cD_{c\omega}$ for each positive constant $c$.
\end{remarks}

\begin{proof}[Proof of Theorem~\ref{T:Dirichlet}]
Let $\omega$ be a weight on $\DD$ such that $\cD_\omega$ is a de Branges--Rovnyak space.
We can suppose that $\omega$ is normalized so that $\|\omega\|_{L^1(\DD)}=1$,
so, by Lemma~\ref{L:Dirichlet}, there exists a holomorphic function $\phi$ on $\DD$ such that 
\eqref{E:phi} holds. Clearly $\phi(0)=0$, so $h(w):=\phi(w)/w$ has a removable singularity at $w=0$. Thus
\begin{equation}\label{E:h}
\int_\DD\frac{1-|w|^2}{|1-z\overline{w}|^4}\omega(z)\,dA(z)=|h(w)|^2
\quad(w\in\DD).
\end{equation}
Extend $\omega$ to the whole of $\CC$ by defining $\omega:=0$ on $\CC\setminus\DD$.
Then $\omega$ is a compactly supported distribution on $\CC$ and,
in the notation of distributions, \eqref{E:h} may be rewritten as
\begin{equation}\label{E:hdist}
\frac{1}{\pi}\Bigl\langle \omega(z),~\frac{1-|w|^2}{|1-z\overline{w}|^4}\Bigl\rangle=|h(w)|^2
\quad(w\in\DD),
\end{equation}
the $\pi$ appearing because $dA$ is normalized area measure on $\DD$.
Now a computation shows that, for all $z,w\in\DD$,
\[
\Delta_z\Bigl(\frac{1-|z|^2}{|1-z\overline{w}|^2}\Bigr)
=4\partial_z\overline{\partial}_z\Bigl(\frac{1-z\overline{z}}{(1-z\overline{w})(1-\overline{z}w)}\Bigr)
=-4\frac{1-|w|^2}{|1-z\overline{w}|^4}.
\]
Hence
\[
-\frac{1}{4\pi}\Bigl\langle\omega(z),~\Delta_z\frac{1-|z|^2}{|1-z\overline{w}|^2}\Bigr\rangle=|h(w)|^2
\quad(w\in\DD),
\]
and so, writing $u:=-(1/4\pi)(1-|z|^2)\Delta\omega$, we have
\[
\Bigl\langle u(z),~\frac{1}{|1-z\overline{w}|^2}\Bigr\rangle=|h(w)|^2
\quad(w\in\DD).
\]
Expanding both sides out in powers of $w$ and $\overline{w}$, 
and equating coefficients of $\overline{w}^jw^k$, we find that
\[
\langle u,z^j\overline{z}^k\rangle=\overline{h}_jh_k,
\]
where $h(w)=\sum_{j}h_jw^j$ is the Taylor development of $h$.
Note also that, from formula \eqref{E:h}, we have $h_0=h(0)=\|\omega\|_{L^1(\DD)}=1$.
Hence, for all integers $j,k\ge0$, we have
\[
\langle u,z^j\overline{z}^k\rangle=\langle u,z^j\rangle\langle u,\overline{z}^k\rangle,
\]
in other words, $u$ is a weakly multiplicative distribution.
We now invoke Theorem~\ref{T:dist}. By that theorem, $u$ is supported at a single point $a\in\CC$.
Thus $u=0$ on $\DD\setminus\{a\}$, whence $\Delta\omega=0$ on $\DD\setminus\{a\}$.
By Weyl's lemma, $\omega$ is (almost everywhere equal to) a function that is harmonic on $\DD\setminus\{a\}$.
Finally, since $\omega\ge0$ on $\DD\setminus\{a\}$, it can be extended  so as to be 
superharmonic on $\DD$ (see e.g.\ \cite[Theorem~3.6.1]{Ra95}).
This completes the proof.
\end{proof}

\begin{remarks}
(i) From here it is a small step to recover explicitly the weights $\omega$ for which $\cD_\omega$
is a de Branges--Rovnyak space. Indeed, knowing that $\omega$ is a positive superharmonic function on $\DD$,
we may write it as the sum of a Green potential and positive harmonic function on $\DD$
(see e.g.\ \cite[Theorem~4.5.4]{Ra95}). 
Thus there exist finite positive Borel measures $\mu$ on $\DD$ and $\nu$ on $\partial\DD$ such that
\[
\omega(z)=\int_\DD \log\Bigl|\frac{1-\overline{\zeta}z}{\zeta-z}\Bigr|\frac{2}{1-|\zeta|^2}\,d\mu(\zeta)
+\int_{\partial\DD}\frac{1-|z|^2}{|\zeta-z|^2}\,d\nu(\zeta)
\quad(z\in\DD).
\]
A computation then shows that, as distributions on $\CC$,
\[
-(1/4\pi)(1-|z|^2)\Delta \omega=\mu+\nu.
\]
But also, from the proof of Theorem~\ref{T:dist}, 
we know that $-(1/4\pi)(1-|z|^2)\Delta\omega$ is supported at a single point.
This implies that one of $\mu,\nu$ is a multiple of a Dirac measure while the other one is zero.
This in turn proves that $\omega$ is a multiple 
of one of the weights displayed in \eqref{E:harmweights} and \eqref{E:superharmweights}.

(ii) If $\cD_\omega$ is a de Branges--Rovnyak space, then $\cD_\omega=\cH(b)$ for some $b$ unique up to multiplication
by a unimodular constant. We can recover $b$ from $\omega$ 
by exploiting the fact that $b(z)=\phi(z)a(z)$, where $\phi$ satisfies \eqref{E:phi} 
and $a$ is the outer function such that $|a(e^{i\theta})|^2=1/(1+|\phi(e^{i\theta})|^2)$ a.e.\ on~$\partial\DD$.
This procedure is carried out in \cite{EKKMR16}, and we do not repeat the details here.

(iii) We have classified those weights $\omega$ on $\DD$ for which $\cD_\omega$ is a de Branges--Rovnyak space
in the sense that $\cD_\omega=\cH(b)$ for some $b$ and also such that \eqref{E:isometry} holds.
One might also ask which weights $\omega$ satisfy $\cD_\omega=\cH(b)$ for some $b$, without requiring that
\eqref{E:isometry} holds.
Rather less is known about this problem. 
Some partial results may be found in \cite{CR13}.
\end{remarks}

\section*{Acknowledgements}
The authors are grateful to Trieu Le for drawing the article  \cite{AR09} to their attention, and to Omar El-Fallah, Karim Kellay and Vadim Ognov for helpful discussions.

\bibliographystyle{plain}      
\bibliography{biblio}

\begin{thebibliography}{10}

\bibitem{Al93}
A.~Aleman.
\newblock The multiplication operator on hilbert spaces of analytic functions.
\newblock Habilitation thesis, Fern Universit\"{a}t, Hagen, 1993.

\bibitem{AR09}
A.~Alexandrov and G.~Rozenblum.
\newblock Finite rank {T}oeplitz operators: some extensions of {D}.
  {L}uecking's theorem.
\newblock {\em J. Funct. Anal.}, 256(7):2291--2303, 2009.

\bibitem{CGR10}
N.~Chevrot, D.~Guillot, and T.~Ransford.
\newblock De {B}ranges-{R}ovnyak spaces and {D}irichlet spaces.
\newblock {\em J. Funct. Anal.}, 259(9):2366--2383, 2010.

\bibitem{CR13}
C.~Costara and T.~Ransford.
\newblock Which de {B}ranges-{R}ovnyak spaces are {D}irichlet spaces (and vice
  versa)?
\newblock {\em J. Funct. Anal.}, 265(12):3204--3218, 2013.

\bibitem{EKKMR16}
O.~El-Fallah, K.~Kellay, H.~Klaja, J.~Mashreghi, and T.~Ransford.
\newblock Dirichlet spaces with superharmonic weights and de
  {B}ranges--{R}ovnyak spaces.
\newblock {\em Complex Anal. Oper. Theory}, 10(1):97--107, 2016.

\bibitem{FM16a}
E.~Fricain and J.~Mashreghi.
\newblock {\em The theory of {$\mathcal H$}({$b$}) spaces. {V}ol. 1}, volume~20
  of {\em New Mathematical Monographs}.
\newblock Cambridge University Press, Cambridge, 2016.

\bibitem{FM16b}
E.~Fricain and J.~Mashreghi.
\newblock {\em The theory of {$\mathcal{H}(b)$} spaces. {V}ol. 2}, volume~21 of
  {\em New Mathematical Monographs}.
\newblock Cambridge University Press, Cambridge, 2016.

\bibitem{Fr98}
F.~G. Friedlander.
\newblock {\em Introduction to the theory of distributions}.
\newblock Cambridge University Press, Cambridge, second edition, 1998.
\newblock With additional material by M. Joshi.

\bibitem{Lu08}
D.~H. Luecking.
\newblock Finite rank {T}oeplitz operators on the {B}ergman space.
\newblock {\em Proc. Amer. Math. Soc.}, 136(5):1717--1723, 2008.

\bibitem{Ra95}
T.~Ransford.
\newblock {\em Potential theory in the complex plane}, volume~28 of {\em London
  Mathematical Society Student Texts}.
\newblock Cambridge University Press, Cambridge, 1995.

\bibitem{Ri91}
S.~Richter.
\newblock A representation theorem for cyclic analytic two-isometries.
\newblock {\em Trans. Amer. Math. Soc.}, 328(1):325--349, 1991.

\bibitem{RS91}
S.~Richter and C.~Sundberg.
\newblock A formula for the local {D}irichlet integral.
\newblock {\em Michigan Math. J.}, 38(3):355--379, 1991.

\bibitem{Sa94}
D.~Sarason.
\newblock {\em Sub-{H}ardy {H}ilbert spaces in the unit disk}, volume~10 of
  {\em University of Arkansas Lecture Notes in the Mathematical Sciences}.
\newblock John Wiley \& Sons, Inc., New York, 1994.
\newblock A Wiley-Interscience Publication.

\bibitem{Sa97}
D.~Sarason.
\newblock Local {D}irichlet spaces as de {B}ranges-{R}ovnyak spaces.
\newblock {\em Proc. Amer. Math. Soc.}, 125(7):2133--2139, 1997.

\bibitem{Sh02}
S.~Shimorin.
\newblock Complete {N}evanlinna-{P}ick property of {D}irichlet-type spaces.
\newblock {\em J. Funct. Anal.}, 191(2):276--296, 2002.

\bibitem{Yo20}
El~H. Youssfi.
\newblock A covariance equation.
\newblock {\em J. Geom. Anal.}, 30(4):3398--3412, 2020.

\end{thebibliography}

\end{document}